\documentclass[11pt]{amsart}

\usepackage{amssymb,amsmath,amsthm}
\usepackage{color}
\usepackage{enumerate}
\newtheorem{theorem}{Theorem}[section]
\newtheorem{example}[theorem]{Example} %[section]
\newtheorem{lemma}[theorem]{Lemma}
\newtheorem{proposition}[theorem]{Proposition}
\theoremstyle{definition}
\newtheorem{definition}[theorem]{Definition}%[theorem]
\theoremstyle{remark}

\newtheorem{remark}[theorem]{Remark}

\numberwithin{equation}{section}

\def\cZ{\mathcal Z}
\def\cY{\tilde{Y}}
\def\bP{\mathbf{P}}
\def\th{{\tt h}}
\def\PM{P_{\mathcal M}}

\def\cM{\mathcal M}
\begin{document}

\setcounter{page}{1} \title{Agler Interpolation Families of Kernels}

\author[M.~T.~Jury, G.~Knese,  and
S.~McCullough]{Michael Jury$^1$, Greg Knese$^2$, 
  and Scott McCullough$^3$}

\address{Department of Mathematics\\
  University of Florida\\
  Box 118105\\
  Gainesville, FL 32611-8105\\
  USA}

\email{mjury@math.ufl.edu}

\address{Department of Mathematics\\
  University of California, Irvine\\
  Irvine, CA 92697-3875}

\email{gknese@uci.edu}

\address{Department of Mathematics\\
  University of Florida\\
  Box 118105\\
  Gainesville, FL 32611-8105\\
  USA}

\email{sam@math.ufl.edu}


 
\thanks{${}^1$Research supported by NSF grant DMS 0701268.  \quad ${}^2$ 
  Research initiated by first two authors at the Fields Institute.
  Travel to UF supported in part by UF. 
  \quad ${}^3$ Research supported by NSF grant DMS 0758306.} 

\date{\today}

\begin{abstract}
  An abstract Pick interpolation theorem for a family of 
  positive  semi-definite kernels on a set $X$ is formulated.
  The result complements those in \cite{Ag} and \cite{AMbook}
  and will subsequently be applied to Pick interpolation on 
  distinguished varieties \cite{JKM}. 
\end{abstract}

\maketitle

\section{Introduction}
  Let $s(z,w)$ denote Szeg\H{o}'s kernel; i.e.,
 \begin{equation*}
   s(z,w) = \frac{1}{1-z\overline{w}},
 \end{equation*}
  for complex numbers $z$ and $w$. The kernel $s$
  is the reproducing kernel for the Hardy space
  $H^2(\mathbb D)$ of functions analytic in the unit
  disc $\mathbb D=\{z\in\mathbb C: |z|<1\}$ with square summable power series. 
  Thus, an analytic function $f:\mathbb D\to \mathbb C$
  with power series expansion 
 \begin{equation*}
   f(z) =\sum_{n=0}^\infty f_n z^n
 \end{equation*}
  is, by definition, in $H^2(\mathbb D)$ if and only if
  $ \sum |f_n|^2$ converges. The Hardy space
  is a Hilbert space with inner product
 \begin{equation*}
  \langle f,g\rangle = \sum_{n=0}^\infty f_n \overline{g_n}.
 \end{equation*}
  Evidently, for a fixed $w$, the function $s_w(z)=s(z,w)$
  is in $H^2(\mathbb D)$ and earns the title of reproducing kernel
  because, for $f\in H^2(\mathbb D)$,
 \begin{equation*}
  f(w) = \langle f, s_w\rangle.
 \end{equation*}

     Szeg\H{o}'s kernel is indispensable to the statement of

 \begin{theorem}[Pick Interpolation]
   Let $n$ be a positive integer. Given points
   $w_1,\dots,w_n;v_1,\dots,v_n \in\mathbb D$,
   there exists an analytic function $f:\mathbb D\to\mathbb D$
   such that $f(w_j)=v_j$ if and only if Pick's matrix,
  \begin{equation*}
     \begin{pmatrix} (1-v_j \overline{v_\ell})s(w_j,w_\ell) \end{pmatrix}
  \end{equation*}
   is positive semi-definite.
 \end{theorem}

  Extensions of the Pick interpolation theorem to domains and settings
  more general than the disc $\mathbb D$ often involve replacing
  the Szeg\H{o} kernel with a family of kernels.  The
  references 
 \cite{AMbook, AMint, RIEOT, R2, DPRS, Ab, B, Pi, CLW, Mint, JK, MP, Sarasonannulus} represent only fraction of the results
  in this direction.  For instance, in Abrahamse's \cite{Ab} interpolation
  theorem on the annulus the Szeg\H{o} kernel is replaced by
  a family of kernels $k^t(z,w)$ -  parameterized  by $t$ in  
  the unit circle $\mathbb T$ - 
  identified by Sarason \cite{Sarasonannulus}. See
  also \cite{AD}.  In a similar vein, the recent constrained Pick interpolation 
  results in \cite{DPRS} \cite{RIEOT}\cite{R2} are stated in terms
  of a family of kernels over the disc canonically  determined by the constraints.
  
  The main result of this paper, Theorem \ref{thm:main} below,
  is a Pick theorem formulated, 
  like the related results in 
  \cite{Ag} and \cite{AMbook}, purely in terms of a collection of kernels.
  The result here has a natural operator algebraic interpretation
  %, a view
  %expanded upon at the end of the paper and
  which is exploited in the proof
  by using the fact that the quotient of
  an operator algebra by a two sided ideal is again an operator algebra.
  (This is a corollary of the Blecher-Ruan-Sinclair (BRS) theorem.
  See \cite{Paulsen} for an exposition of the BRS theorem and
  the related topics of completely positive maps, Arveson's extension
  theorem, and Stinespring's representation theorem.)
  In forthcoming work \cite{JKM}, 
  Theorem \ref{thm:main} is applied to produce a Pick interpolation
  theorem on distinguished varieties 
  \cite{AMvarieties} \cite{AMint}.

  The statement of the main result requires the
  notion of a (positive semi-definite) matrix-valued kernel.
  Let $M_n$ denote the $n\times n$ matrices
  with complex entries. An $M_n$-valued  kernel on a set $X$ 
  is a function $k:X\times X\to M_n$ which is positive
  semi-definite in the sense that, for every finite subset
  $F\subset X$, the (block) matrix
 \begin{equation*}
    \begin{pmatrix} k(x,y) \end{pmatrix}_{x,y\in F}
 \end{equation*}
  is positive semi-definite.

\begin{definition}
 \label{def:kernel-family}
   Fix a set $X$ and a sequence $\mathcal K=(\mathcal K_n)$
   where each $\mathcal K_n$ is a set of $M_n$-valued kernels
   on $X$. 
%   of kernels on $X$ such that each $k\in\mathcal K_n$ is an $n\times n$ 
%   matrix valued kernel; i.e.,
%    $k:X\times X \to M_n$ is positive semi-definite.

   The collection $\mathcal K$ is an {\it Agler interpolation
   family of kernels} provided:
 \begin{itemize}
      \item[(i)] if $k_1 \in \mathcal K_{n_1}$ and $k_2\in\mathcal K_{n_2}$, 
       then  $k_1\oplus k_2 \in \mathcal K_{n_1+n_2}$;
   \item[(ii)] if $k\in\mathcal K_n$, $z\in X$, 
         $\gamma\in \mathbb C^n$, and $\gamma^* k(z,z)\gamma\neq 0$,  then
     there exists an $N$, a kernel  $\kappa \in\mathcal K_N$,
     and a function $G:X\to M_{n,N}$ such that     
    \begin{equation*}
      k^\prime(x,y)  :=  k(x,y)
          - \frac{k(x,z)\gamma \gamma^* k(z,y)}{\gamma^*k(z,z)\gamma} 
          = G(x) \kappa(x,y)G(y)^*;
    \end{equation*}
   \item[(iii)] for each finite $F\subset X$
    and for each  $f:F\to \mathbb C$,  there is a $\rho>0$ 
       such that, for each
       $k\in\mathcal K$,
    \begin{equation*}
      F\times F \ni \, \mapsto (\rho^2 -f(x)f(y)^*)k(x,y) 
    \end{equation*}
      is a positive semi-definite kernel on $F$; and
    \item[(iv)] for each $x\in X$ there is a $k\in\mathcal K$ 
      such that $k(x,x)$
      is nonzero. % (and positive semi-definite).
 \end{itemize}
\end{definition}

\iffalse
\begin{remark}\rm
  The intersection of interpolation families is not necessarily
  an interpolation family; however, the intersection of
  closed interpolation families is a closed interpolation family.
\end{remark}
\fi

\begin{remark}\rm
 \label{rem:restrict}
 Given $Y\subset X$ and a kernel
  $k:X\times X \to M_n$, let $k|_Y=k|_{Y\times Y}$.
  Thus $k|_Y$ is a kernel on $Y$. 
 If $\mathcal K$ is an Agler interpolation 
  family of kernels (on $X$), then  $\mathcal K_Y,$ the collection
 of kernels of the form $k|_Y$ for $k\in \mathcal K$,
 is an Agler interpolation family of kernels (on $Y$).
\end{remark}

  %(Worry about getting $k^\prime$ right??)

\begin{theorem}
 \label{thm:main}
  Suppose $\mathcal K$ is an Agler interpolation family of kernels on $X$.
  Further suppose $Y\subset X$ is finite,
  $g:Y\to \mathbb C$ and $\rho\ge 0$.
  If for each $k\in\mathcal K$ the kernel
 \begin{equation}
  \label{eq:thm-main-Y}
    Y\times Y \ni (x,y) \to (\rho^2 -g(x)g(y)^*)k(x,y) 
 \end{equation} 
  is positive semi-definite, then there exists $f:X\to \mathbb C$ 
  such that $f|_Y = g$ and for each $k\in\mathcal K$ the kernel
 \begin{equation}
  \label{eq:thm-main-X}
      X\times X \ni \to (\rho^2 -f(x)f(y)^*)k(x,y) 
 \end{equation}
      is positive semi-definite.
\end{theorem}

\begin{remark}\rm
 \label{rem:scalar-int}
   Theorem \ref{thm:main} is stated for scalar-valued interpolation; i.e.,
   the functions $f$ and $g$ take values in $\mathbb C$ as opposed
   to $M_n$.  In this case it suffices to consider a collection
   of scalar kernels canonically associated with $\mathcal K$ giving
   a result more in line with that found in \cite{Ag} and \cite{AMbook}.
   Some details are provided in Section \ref{sec:scalar}. 
\end{remark}

  In the forthcoming paper \cite{JKM}, Theorem \ref{thm:main} is
  applied to yield a Pick interpolation theorem for distinguished
  varieties. There are similarities to interpolation on multiply
  connected domains and the case of the annulus is discussed
  in Section \ref{sec:example}, where the role of item (ii) 
  of Definition \ref{def:kernel-family} becomes apparent.

\iffalse 
\begin{remark}\rm
 \label{rem:varities}
   In our forthcoming application to distinguished varieties the full
   force of condition (ii) of Agler interpolation family. Rather
   the kernel $k^\prime$ is already in the interpolation family. 
   This situation is analogous to that for interpolation 
   on multiply connected domains and the case of the annulus
   is discussed in Section \ref{sec:example}.
\end{remark}\rm
\fi

\section{Operator Theoretic Preliminaries}
  The operator theoretic approach to interpolation 
  associates to a positive semi-definite matrix-valued kernel $k$
  a Hilbert space $H^2(k)$. Functions satisfying, for this given
  $k$, the positivity condition of item (iii) of Definition
  \ref{def:kernel-family} determine bounded operators on $H^2(k)$.

 \subsection{The Hilbert Space $H^2(k)$}
  To a positive semi-definite kernel $k:X\times X\to M_n$, 
  there is associated a Hilbert space $H^2(k)$ so that
  in the case that $k$ is positive definite and 
  $X$ is finite,
  $H^2(k)$ is, as a set,  all functions $F:X\to \mathbb C^n$.
  To construct $H^2(k)$, define a semi-inner product
  on functions
  $F,G:X\to \mathbb C^n$ of the form
 \begin{equation*}
  \begin{split}
    F = & \sum_{x\in X} k(\cdot,x) F_x, \\
    G = & \sum_{x\in X} k(\cdot,x) G_x,
  \end{split}
 \end{equation*}
  %where $F_x,G_x\in\mathbb C^n,$ 
  by
 \begin{equation*}
  \langle F,G\rangle = \sum_{x,y\in X} \langle k(x,y)F_y,F_x\rangle.
 \end{equation*}
   Let $H^2(k)$ denote the Hilbert space obtained by quotienting
   out null vectors and then forming the completion
   of the resulting pre-Hilbert space.
   When $X$ is finite the
   quotient is finite dimensional and hence already complete. 
   If moreover, $k$ is positive definite, then the set
   of null vectors is trivial. 
 
  Condition (ii) in Definition 1.2 has a natural interpretation 
  in terms of $H^2(k)$:  if $\mathcal N$ is the 
  subspace of $H^2(k)$ spanned by the nonzero vector 
  $k(\cdot, z)\gamma$, then $k^\prime$ is the 
  reproducing kernel for $\mathcal N^\bot$.  Indeed, we have
 \begin{equation*}
    P_{\mathcal N} =  \frac{k(\cdot,z)\gamma (k(\cdot,z)\gamma)^*}
          {\langle k(\cdot,z)\gamma,k(\cdot,z)\gamma \rangle }.
 \end{equation*}
   Hence,
 \begin{equation*}
  \begin{split}
   \langle P_{\mathcal N} k(\cdot,y)v,k(\cdot,x)u \rangle 
     = & \frac{\langle k(\cdot,y)v,k(\cdot,z)\gamma\rangle 
         \langle k(\cdot,z)\gamma,k(\cdot,x)u\rangle }
            {\langle k(z,z)\gamma,\gamma \rangle } \\
     = &\frac{\langle k(z,y)v,\gamma \rangle \langle k(x,z)\gamma,u \rangle}
              {\langle k(z,z)\gamma,\gamma\rangle}  \\
     = & u^* \frac{k(x,z)\gamma\gamma^*k(z,y)}
         {\langle k(z,z)\gamma,\gamma \rangle} v.
  \end{split}
 \end{equation*}
  Thus, letting $\mathcal M = H^2(k)\ominus \mathcal N$
   and using the notation of item (iii) in Definition \ref{def:kernel-family},
 \begin{equation*}
  \begin{split}
   \langle \PM k(\cdot,y)v,&k(\cdot,x)u \rangle \\
     = & \langle k(\cdot,y)v,k(\cdot,x)u\rangle
        - \frac{\langle k(x,z)\gamma\gamma^*k(z,y)v,u\rangle}
         {\langle k(z,z)\gamma,\gamma \rangle} \\
     = &  \langle k(x,y)v,u\rangle 
            - \frac{\langle k(x,z)\gamma\gamma^*k(z,y)v,u\rangle} 
         {\langle k(z,z)\gamma,\gamma \rangle} \\
     = & \langle k^\prime(x,y) v,u\rangle.
  \end{split}
 \end{equation*}
   
  Assuming $k$ is a member of an Agler interpolation family
  $\mathcal K$, then, by item (iii) of
  Definition \ref{def:kernel-family}
  there is an $N$, a $\kappa\in\mathcal K_N$,
  and a  function $G:X\to M_{n,N}$ such that
 \[
  \langle \PM k(\cdot,y)v, k(\cdot,x)u \rangle
    = \langle G(x)\kappa(x,y)G(y)^*v,u\rangle.
 \]

 \begin{lemma}
  \label{lem:compress-abstract}
   Let $\mathcal K$ be an Agler interpolation family of
   kernels on a finite set $X$. Suppose $k\in\mathcal K,$ 
   $Z\subset X$ and for each $z\in Z$ there
   is an associated subspace $\mathcal J_z\subset \mathbb C^n.$
   Let $\mathcal G_z = k(\cdot,z)\mathcal J_z,$ 
   % For $z\in \mathcal Z$, let $\mathcal G_z =k(\cdot,z)\mathbb C^n$ and
   let $\mathcal N =\sum \mathcal G_z \subset H^2(k),$ 
   and let $\mathcal M = H^2(k)\ominus \mathcal N$.
   There is an $N$, a kernel $\kappa \in \mathcal K_N$,
   and a function $G:X\to M_{n,N}$ such that  
  \begin{equation}
   \label{eq:compress}
    \langle  \PM k(\cdot,y)v,k(\cdot,x)u \rangle  
        = \langle G(x) \kappa(x,y)G(y)^*u,v\rangle .  
  \end{equation}

   Moreover, there is a positive $M_n$-valued kernel $k^\prime$ 
   such that, for $v,w\in \mathbb C^n$, 
  \begin{equation}\label{E:moreover}
   \langle k^\prime(x,y)u,v\rangle = 
      \langle \PM k(\cdot,y)u,k(\cdot,x)v\rangle.
  \end{equation}

   Finally, the mapping $W:\mathcal M \to H^2(\kappa)$ defined by
  \[
    W \PM k(\cdot,y)u= \kappa(\cdot,y)G(y)^*u
  \]
   is (well defined and) an isometry.
 \end{lemma}

 \begin{proof}
   Equation \eqref{eq:compress} follows by an induction argument based
   on the computation preceding the proof.  The right hand side of
   equation (\ref{E:moreover}) determines a (positive semi-definite)
   kernel.  Finally, that $W$ is an isometry follows immediately
   from equation \eqref{eq:compress}.
 \end{proof}

\begin{lemma}
 \label{lem:Qx-defined}
   Suppose $X$ is a finite set and 
   $k:X\times X\to M_n$ is a (positive semi-definite) kernel.
   If for each $f:X\to \mathbb C$ there exists a $\rho>0$
   such that
 \begin{equation*}
   X\times X \ni (x,y) \to (\rho^2 -f(x)f(y)^*)k(x,y)
 \end{equation*}
   is positive semi-definite, then, for each $x\in X$ the
   mapping 
    \begin{equation*}
     Q_x k(\cdot,y)v = \begin{cases} k(\cdot,x)v & \ \ y=x; \\
                               0 & \ \ y\ne x.
      \end{cases}
  \end{equation*}
   determines a well defined mapping $Q_x:H^2(k)\to H^2(k)$.
\end{lemma}

\begin{proof}
%   That $Q_x$ is in fact well defined follows from 
%   item (iii) of Definition \ref{def:kernel-family}.
%   To verify this claim, let $f$ denote the indicator
%   function of $\{x\}$. Item (iii) postulates the
%  existence of a $\rho>0$ such that the block matrix
   Given $x$, let $f$ denote the indicator function 
   of the subset $\{x\}$ of $X$. By hypothesis, there exists
   $\rho>0$ such that 
 \begin{equation*}
    A=((\rho^2 - f(y)f(z)^*)K(y,z))_{y,z\in X}
 \end{equation*}
  is positive semi-definite. Consequently, 
  if $\sum_y k(\cdot,y)v_y =0,$ then, letting
  $v$ denote the vector with $y$-th entry $v_y$, 
 \begin{equation*}
   0\le \langle Av,v\rangle =\| \sum k(\cdot,y)v_y \|^2 - 
      \langle k(x,x)v_x,v_x\rangle \le 0,
 \end{equation*}
  from which it follows that $k(\cdot,x)v_x=0$. 
\end{proof}

\subsection{The algebra $H^\infty(k)$} 
 \label{subsec:H-infty-k}
  Let $k$ be a positive semi-definite $M_n$-valued kernel on $X$
  and suppose for each $f:X\to\mathbb C$ there
  is a $\rho>0$ such that
 \begin{equation*}
  X\times X \ni (x,y)\mapsto (\rho^2-f(x)f(y)^*)k(x,y)
 \end{equation*}
  is a positive semi-definite kernel on $X$.  
  Let $H^\infty(k)$ denote the set of functions $f:X\to \mathbb C$
  endowed with the norm,
 \begin{equation*}
   \|f\|_k =\inf\{\rho>0 : (\rho^2 - f(x)f(y)^*)k(x,y)\succeq 0 
        \mbox{ for all } k\in\mathcal K\}.
 \end{equation*}
   Here $\succeq 0$ means the relevant kernel is positive semi-definite.

  An  element $f$ of $H^\infty(k)$ is identified with the
  operator $M_f:H^2(k)\to H^2(k)$ whose adjoint is determined
  by $M_k(f)^* k(\cdot,z)h= f(z)^* k(\cdot,z)h$. Indeed,
 \begin{equation*}
  \|M_k(f)^*\|_k =\|f\|_k.
 \end{equation*}
   Hence $M_k:H^\infty(k)\to \mathcal B(H^2(k))$ defined by
   $f\mapsto M_k(f)$ is an isometric unital representation. 
   Moreover, viewing $H^\infty(k)$ as a subalgebra of
   $\mathcal B(H^2(k))$ determines an operator algebra
   structure on $H^\infty(k)$.
 
  \begin{lemma}
   \label{lem:represent}
     %Suppose $k$ is a positive definite kernel
     %on the finite set $X$ with values in $M_n$
    % so that $k:X\times X\to M_n$. 
     Suppose $X$ is finite.
     If $\mathcal H$ is a Hilbert space and 
     $\tau:H^\infty(k)\to \mathcal B (\mathcal H)$
     is a completely contractive unital representation, 
     then there is a Hilbert space $\mathcal E$ and an isometry 
     $V:\mathcal H\to \mathcal E \otimes H^2(k)$
     such that
   \begin{equation*}
      \tau(f)= V^* (I \otimes M_k(f)) V.
   \end{equation*}
  \end{lemma}

 \begin{proof}
   Identify $H^\infty(k)$ with the subspace
   $\{M_k(f): f\in H^\infty(k)\}$ of  $\mathcal B(H^2(k))$.
   Since $\tau$ is completely contractive and unital, it
   extends to a completely contractive unital map 
   $\Phi:\mathcal B(H^2(k))\to \mathcal B(\mathcal H)$.
   By Stinespring's representation theorem, there exists
   a Hilbert space $\mathcal L,$ an isometry $V:\mathcal H\to \mathcal L$,
   and a representation $\pi:\mathcal B(H^2(k))\to \mathcal B(\mathcal L)$
    such that 
  \begin{equation*}
     \Phi(T)=V^*\pi(T)V.
  \end{equation*}
   In particular, for $f\in H^\infty(k)$, we have 
   $\tau(f)=V^* \pi(M_k(f)) V.$

   Since $H^2(k)$ is finite dimensional (as $X$ is finite), 
   $\pi$ is a multiple of
   the identity representation; i.e., up to unitary equivalence,
   $\pi(T)= I\otimes T$, and under this identification there
   is a Hilbert space $\mathcal E$ such that
   $\mathcal L=\mathcal E\otimes H^2(k)$. 
 \end{proof}

\section{The Proof for finite $X$}
  In this section we prove Theorem \ref{thm:main}
  first under the added hypothesis that $X$ is
  a finite set.  Accordingly, until Section
  \ref{subsec:infiniteX}, assume that $X$ is finite.

\subsection{Representations of quotients}
 \label{subsec:quotients}
 %  Suppose $k$ is a positive semi-definite
 %  kernel on $X$ for which $H^\infty(k)$ is
 %  defined. Fix a subset $Y$ of $X$ and
 %  let $I$ denote the ideal of functions
 %  in $H^\infty(k)$ which vanish on $Y$.
 %  Representations of the quotient
 %  $H^\infty(k)/I$ are naturally identified
 %  with representations of $H^\infty(k)$
 %  which annihilate $I$.
 %  Given a collection $\mathcal K$ of kernels on 
 %  a set $X$ and a subset $Y$ of $X$, let
 %  $\mathcal K|_Y$ denote set of restriction 
 %  $k|_{Y\times Y}=k|_Y$ for $k\in\mathcal K$.
   Given $f:X\to \mathbb C,$ let $\cZ(f)$ denote the zero
   set of $f$. The statement of the following lemma 
   also uses the notation $\mathcal K|_Y$ from
   Remark \ref{rem:restrict}

  \begin{lemma}
   \label{lem2}
   Suppose
 \begin{itemize}
    \item[(i)]  $\mathcal K$ is an Agler interpolation family on the 
             finite set $X;$
    \item[(ii)]  $k\in\mathcal K_n;$
    \item[(iii)]  $\mathcal H$ and $\mathcal E$ are Hilbert spaces, 
       and  $V:\mathcal H\to \mathcal E\otimes H^2(k)$
         is an isometry;
    \item[(iv)]   $\sigma:H^\infty(k)\to \mathcal B(\mathcal H)$ given by
      \begin{equation*}
         H^\infty(k) \ni f \mapsto V^* (I\otimes M_k(f)) V
      \end{equation*}
        is a (unital) representation; and
    \item[(v)]  $Y\subset X.$ 
  \end{itemize}

     If $\sigma(g)=0$ whenever $Y\subset \cZ(g),$ then, for
     each $\psi\in H^\infty(k),$  
   \begin{equation*}
     \|\sigma(\psi)^*\|\le \sup\{ \|M_{\kappa}(\psi|_{Y}) \| :
        \kappa \in \mathcal K|_Y \}.
   \end{equation*}
  \end{lemma}

 \begin{remark}\rm
  \label{rem:lem2}
     Note  $\sigma(\psi)^*$ depends only upon $\psi|_Y$. In
     fact, $\sigma$ induces a representation 
     $\tilde{\sigma}:H^\infty(k)/I \to \mathcal B(\mathcal H)$, where
     $I$ is the ideal of functions in $H^\infty(k)$ which vanish
     on the complement, $\cY,$ of $Y$ in $X$. 
 \end{remark}

  \begin{proof}
  Fix $\psi\in H^\infty(k)$ and $\epsilon>0$. Choose unit vectors
  $h,\gamma$ in $\mathcal H$ such that
 \begin{equation}
  \label{lem2:eq3}
  \|\sigma(\psi)^*\| \le \langle \sigma(\psi)^* h,\gamma\rangle +\epsilon.
 \end{equation}

  Because $X$ is a finite set, there exists a finite dimensional
  subspace $\mathcal E_0$ of $\mathcal E$ such that
  $V\gamma \in \mathcal E_0 \otimes H^2(k)$. Let $K$
  denote the kernel $K:X\times X\to \mathcal B(\mathcal E_0) \otimes H^2(k)$
  defined by
 \begin{equation*}
   K(x,y) e\otimes v = e\otimes k(x,y)v.
 \end{equation*}
   Since $\mathcal K$ is closed with respect to direct
  sums, $K\in\mathcal K$.  Indeed, $K$ is the
  direct sum of $k$ with itself $m$ times, where
  $m$ is the finite dimension of $\mathcal E_0$.
  Let $N=mn$ and view $K:X\times X\to \mathbb C^N$.
  Summarizing,  $H^\infty(k)=H^\infty(K)$
  (as operator algebras), 
  $\mathcal E\otimes H^2(k)$ is canonically identified with
  $H^2(K) \oplus (\mathcal E_0^\perp \otimes H^2(k)),$
  and $V\gamma \in H^2(K)$.

  Let $\bP$ denote the projection onto $H^2(K)$. 
  Thus $\bP = P_{\mathcal E_0}\otimes I,$ from
  which it follows that 
  the subspace $H^2(K)$ reduces
  $(I_\mathcal E\otimes M_k(\varphi)^*)$ 
    for each $\varphi\in H^\infty(k).$
  Thus,
  for $\th \in \mathcal H$,  
 \begin{equation}
  \label{lem2:eq2}
   \begin{split}
     \langle \sigma(\varphi)^* \th , \gamma \rangle 
     = & \langle V^* (I_{\mathcal E}\otimes M_k(\varphi)^*) V\th,
                \gamma \rangle \\
     = & \langle \bP (I_{\mathcal E}\otimes M_k(\varphi)^*) V\th,
           V\gamma \rangle \\
     =&  \langle (I_{\mathcal E_0}\otimes M_k(\varphi)^*) \bP V\th,
           V\gamma \rangle \\
     = & \langle V^* M_K(\varphi)^* \bP V\th, \gamma\rangle,
  \end{split}
 \end{equation}
   where $V\gamma = \bP V\gamma$ was used in the second equality.

%    First suppose $\mathcal E$ is finite dimensional.  
%    Define the kernel $K:X\times X \to \mathcal E \otimes \mathbb C^n$
%   by 
% \begin{equation*}
%   K(x,y) e\otimes v = e\otimes k(x,y)v.
% \end{equation*}
  
    Because of item (iii) in the definition of interpolation family
   and Lemma \ref{lem:Qx-defined}, 
    for $x\in X$, 
  \begin{equation*}
     Q_x K(\cdot,y)v = \begin{cases} K(\cdot,x)v & \ \ y=x; \\
                               0 & \ \ y\ne x
      \end{cases}
  \end{equation*}
   determines a bounded operator $Q_x:H^2(K)\to H^2(K)$.

   Next observe $Q_x^2 =Q_x,$  the range of 
   $Q_x$ is $ [K(\cdot,x)v: v\in\mathbb C^N],$
   there is the (non-orthogonal) resolution  $I=\sum_x Q_x,$ and
  \begin{equation}
   \label{lem2:eq1}
      M_K(\varphi)^* Q_x =\varphi(x)^* Q_x
  \end{equation}
      for $\varphi \in H^\infty(k)$. 

    For $x\in X$, let
  \begin{equation*}
     \mathcal G_x = Q_x \bP V\mathcal H.
  \end{equation*}
     Observe $\mathcal G_x$ is invariant for  
    $\{M_K(\psi)^* : \psi\in H^\infty(K)\}$ because
    of equation \eqref{lem2:eq1}.
    Thus  $\mathcal G_{\cY}=\sum_{z\notin Y} \mathcal G_z$
    is 
   invariant for $\{M_K(\psi)^* : \psi\in H^\infty(k)\}.$
   Let $\mathcal M=H^2(K)\ominus \mathcal G_{\cY}$.

    If $g\in H^\infty(k)$ and  $Y\subset \cZ(g)$, 
    and if $\th \in \mathcal H$, then
  \begin{equation*}
   \begin{split}
    0 = & \langle  \sigma(g)^* \th, \gamma \rangle \\
      %= & \langle \sigma(g)^* V\th, V\gamma \rangle \\
      %= & \langle \bP \sigma(g)^* V \th, V\gamma \rangle \\
      = & \langle M_K(g)^* \bP V\th,V\gamma \rangle \\
      = & \langle \sum_x g(x)^* Q_x \bP V\th, V\gamma \rangle \\
      = & \langle \sum_{z\notin Y} g(z)^* Q_z \bP V\th, V\gamma \rangle. 
   \end{split}
  \end{equation*}
   The first equality follows from the hypothesis on  $\sigma$
   which gives $\sigma(g)=0$; the second uses
   equation \eqref{lem2:eq2}; 
   the third uses equation \eqref{lem2:eq1}
   and $I=\sum Q_x$; and  
   the fourth equality
   from the fact that $g(y)=0$ for $y\in Y$.
    Fix a $z_0\notin Y$ and use 
    item (iii) in the definition of interpolation
    family to choose $g\in H^\infty(k)$ such that $g(z_0)=1$ and
    $g(x)=0$ otherwise to obtain
  \begin{equation*}
     0 = \langle Q_{z_0} \bP V\th,V\gamma \rangle .
  \end{equation*}
     Thus, $ V\gamma$ is orthogonal to each
     $\mathcal G_{z_0}$ and therefore to $\mathcal G_{\cY}.$ 
    % Hence, the range of $V$ lies in $\mathcal M$.
     Hence $V\gamma \in \mathcal M.$

   Since $\PM Q_z \bP V\mathcal H=0$ for $z\notin Y$,
   if  $\th \in \mathcal H$ and $\bP V \th$ 
   is written as 
 \begin{equation*}
   \bP V\th = \sum_{y\in X} K(\cdot,y)v_y,
 \end{equation*}
  then, for $z\notin Y,$
 \begin{equation}
  \label{lem2:eq6}
    \PM K(\cdot,z)v_z =0.
 \end{equation}
 %  since $K(\cdot,z)v_z = Q_z \bP V \th$. 
  In particular,
 \begin{equation}
  \label{lem2:eq7}
   \PM \bP V\th = \sum_{y\in Y} \PM K(\cdot,y)v_y.
 \end{equation}
  Thus,
  with $\mathcal L$ equal to the 
  span of $\{K(\cdot,y)v: y\in Y, \ \ v\in \mathbb C^N\}$,
  it follows that $\PM \bP V \mathcal H\subset \mathcal L$.

  From Lemma \ref{lem:compress-abstract}
  there is an $M,$  a kernel $\kappa \in \mathcal K_M$,
  and a function $G:X\to M_{N,M}$ such that
 \begin{equation}
  \label{eq:K-prime-kappa}
    \langle \PM K(\cdot,y)v,K(\cdot,x)u\rangle
     =  \langle \kappa(x,y)G(y)^* v, G(x)^*u \rangle.
 \end{equation}
  In particular, the map
  $W:\mathcal L \to H^2(\kappa|_Y)$ defined by
  $W \PM K(\cdot,y)v = \kappa(\cdot,y)G(y)^* v$
  is (well defined and) an isometry.

    Returning to the  vector $h\in\mathcal H$
    in equation \eqref{lem2:eq3}, 
    there exists $h_x,\gamma_x \in \mathbb C^N$ such that
 \begin{equation*}
  \begin{split}
   \bP V h=& \sum_{x\in X} Q_x \bP Vh = \sum K(\cdot,x)h_x \\
   V \gamma = & \sum_{x\in X} Q_xV\gamma =  \sum K(\cdot,x) \gamma_x.
  \end{split}
 \end{equation*}
%   Let
% \begin{equation*}
%  \begin{split}
%    h^\prime = & \PM\bP Vh = \sum_{y\in Y} \PM K(\cdot,y)h_y \\
%    \gamma^\prime = & PV\gamma = V\gamma=
%           \sum_{y\in Y} \PM K(\cdot,y) \gamma_y.
%  \end{split}
% \end{equation*}
  Note that, since $h$ and $\gamma$ are unit vectors, that
  $\|\bP V h \|\le 1$ and $\|V \gamma  \|=1$.

  With these notations and for $\varphi\in H^\infty(k)$,  
 \begin{equation*}
  \begin{split}
   \langle \sigma(\varphi)^* h, \gamma \rangle 
   = & \langle M_K(\varphi)^* \bP Vh,V\gamma \rangle \\
   = & \sum_{y \in X} \langle \varphi(y)^* K(\cdot,y)h_y, 
       \PM V\gamma \rangle \\
   = & \sum_{x,y\in X} \langle \varphi(y)^* K(\cdot,y) h_y,
            \PM K(\cdot,x)\gamma_x \rangle \\
   = & \sum_{x,y\in X}  \langle \varphi(y)^* \PM K(\cdot,y)h_y,
               \PM K(\cdot,x)\gamma_x  \rangle \\
   = & \sum_{x,y \in Y} \langle \varphi(y)^* \PM K(\cdot,y)h_y, 
        \PM K(\cdot,x)\gamma_x  \rangle\\
   = & \sum_{x,y \in Y} \langle \varphi(y)^* W \PM K(\cdot,y)h_y, 
        W \PM K(\cdot,x)\gamma_x  \rangle\\
   = & \sum_{x,y\in Y} \langle \varphi(y)^* \kappa(\cdot,y)G(y)^* h_y,
         \kappa(\cdot,x)G(x)^* \gamma_x \rangle \\
   = & \langle M_{\kappa|_Y}(\varphi|_Y)^* \sum_{y\in Y}
       k(\cdot,y)G(y)^* h_y, \sum_{x\in Y} k(\cdot,x)G(x)^* \gamma_x\rangle\\
   = & \langle M_{\kappa|_Y}(\varphi|_Y)^* W \PM \bP V h, W V\gamma \rangle.
  \end{split}
 \end{equation*}
  Here the first equality follows from
  the definition of $\sigma$; the
  second uses equation \ref{lem2:eq1} and
  $V\gamma \in \mathcal M$;
  the fifth uses equation \eqref{lem2:eq6}; the sixth that
  $W:\mathcal L\to H^2(\kappa|_Y)$ is an isometry;  the seventh
  the definition of $W$; and finally the last equality uses
  both the definition of $W$ and equation \eqref{lem2:eq7}.

\iffalse %%%%%%%%%%%%%% 
 The equation \eqref{eq:K-prime-kappa} implies $W$ is 
  well defined and an isometry since,
 \begin{equation*}
  \begin{split}
  \langle W K^\prime(\cdot,x) v, WK^\prime(\cdot,y) w\rangle
    = & \langle \kappa(\cdot,x) G(x)v, \kappa(\cdot,y) G(y) w \rangle \\
    = & \langle G(y)^* \kappa(y,x)G(x)v,w \rangle \\
    = & \langle K^\prime(y,x)v,w \rangle \\
    = & \langle K^\prime(\cdot,x) v, K^\prime(\cdot,y) w \rangle.
  \end{split}
 \end{equation*}
   Moreover,
 \begin{equation*}
   W^* M_{\kappa|_Y}(\varphi|_Y)^* W
      =  M_{K^\prime|_Y}(\varphi|_Y)^*.
 \end{equation*}

  It now follows that
 \begin{equation}
  \label{eq:K-Kprime}
   \begin{split}
    \|\sigma(\varphi)^* \| -\epsilon 
      \le &  |\langle \sigma(\varphi)^* h, g\rangle | \\
      = & |\langle M_{K^\prime|Y}(\varphi|_Y)^* h^\prime,
             W \gamma^\prime \rangle_{H^2(K^\prime|_Y)} | \\
      \le & \| M_{K^\prime|Y}(\varphi|_Y)^* \| \|h^\prime \|\|\gamma\| \\
      \le &  \| M_{K^\prime|Y}(\varphi|_Y)^* \|.
    \end{split}
  \end{equation}
  Since $\epsilon>0$ is arbitrary,
 \begin{equation*}  
   \|\sigma(\psi)^*\| \le  \|M_{K^\prime|_Y}(\varphi|_Y)^*\|.
 \end{equation*}

\fi %%%%%%%%%%%%%%%%%%%%%%%%%%

   Hence,
 \begin{equation*}
  \begin{split}
    \|\sigma(\varphi)^* \| -\epsilon 
      \le &  |\langle \sigma(\varphi)^* h, g\rangle | \\
      = & |\langle M_{\kappa|_Y}(\varphi|_Y)^*  W\PM \bP V h, 
            V\gamma \rangle | \\
      \le & \| M_{\kappa|_Y}(\varphi|_Y)^* \| \ \|  W\PM \bP V h \| 
         \ \| W V\gamma \| \\
      \le & \| M_{\kappa|_Y}(\varphi|_Y)^* \| \  \|h\| \ \|\gamma \|.
    \end{split}
  \end{equation*}
   and the proof is complete.
\end{proof}

\subsection{The end of the  proof for finite $X$}
 \label{subsec:finiteX}
  In this subsection we complete the  proof of
  Theorem \ref{thm:main}
  in the case that $X$ is finite, in which case 
  there exists $m$ and $x_1,\dots,x_m$ such that
  $Y=X\setminus \{x_1,\dots,x_m\}.$
  Fix $g:Y\to\mathbb C$.  Suppose $\psi:X\to \mathbb C$ 
  extends $g$ so that $g=\psi|_Y$,   
  and define,
 \begin{equation*}
   \rho = \sup\{ \|M_{k|_Y}(\psi|_Y)\|:
        k\in \mathcal K \}.
 \end{equation*}
  Note that $\rho$ depends only upon $g$.

  Let $\tilde{k}$ be a given element of $\mathcal K$.  Let $\mathcal
  I_{\tilde k}$ denote the ideal of functions in $H^\infty(\tilde{k})$
  which vanish $Y$.  The quotient $H^\infty(\tilde{k})/\mathcal
  I_{\tilde{k}}$ is a unital operator algebra and hence (by the BRS
  theorem) it has a completely contractive unital representation
  $\tau$ on a Hilbert space $\mathcal H$.

  The quotient mapping
 \begin{equation*}
     \pi:H^\infty(\tilde{k})\to H^\infty(\tilde{k})/I_{\tilde{k}}
 \end{equation*}
   is completely contractive and unital.  Thus, 
   $\sigma = \tau \circ \pi:H^\infty(\tilde{k}) \to \mathcal B(\mathcal H)$
  is a completely contractive unital representation. 
  Further, because $\tau$ is a (complete) isometry,
\[
   \|\pi(\psi)\| =\|\sigma(\psi)\|
\]
  for $\psi\in H^\infty(\tilde{k})$.  
  Since $\pi$ is a unital completely contractive  
  representation of $H^\infty(\tilde{k})$,
  $\pi$ has the form given in Lemma \ref{lem:represent}.
  Hence, Lemma \ref{lem2} applies to give
 \[
   \|\pi(\psi)\|\le \rho.
 \]
  
   Suppose now that $\rho^\prime>\rho$. Then, by the definition of
   the quotient norm, there exists a $\varphi$ such that 
   $\pi(\varphi)=\pi(\psi)$
   and so that
 \begin{equation}
  \label{eq:fip}
    X\times X \ni (x,y)\to  [(\rho^\prime)^2 -\varphi(x)\varphi(y)^*] 
       \tilde{k}(x,y) 
 \end{equation}
   is positive semi-definite.

   Consider the set 
 \begin{equation*}
      C_{\tilde{k}} = \{(\varphi(x_1),\dots, \varphi(x_m))
        : \pi(\varphi)=\pi(\psi)
         \mbox{ and equation \eqref{eq:fip} holds}\} \subset \mathbb C^m.
 \end{equation*}
   From above $C_{\tilde{k}}$ is nonempty. It is also closed,
   and item (iv) in the definition of interpolation family 
   implies it is bounded.
   Because $\mathcal K$ is closed with respect to direct sums, the collection 
   $\{C_{\tilde{k}}:\tilde{k} \in\mathcal K\}$ has the finite intersection
   property. Hence, there exists a $\varphi$ such that
   $\varphi|_Y = g$ and, for each $k\in\mathcal K$, the kernel
 \begin{equation}
  \label{rhoprime}
     X\times X \ni (x,y)\to  [(\rho^\prime)^2 -\varphi(x)\varphi(y)^*] 
        k(x,y) 
 \end{equation}
  is positive semi-definite.  

  To finish the proof, choose a sequence $\rho_\ell>\rho$ 
  converging to
  $\rho$. There exists $\varphi_\ell$ such that the kernel
  in equation \eqref{rhoprime}, 
  with $\varphi_\ell$ in place of $\varphi$
  and $\rho_\ell$ in place of $\rho^\prime$, is positive semi-definite.
  Because  $\varphi_\ell$ is uniformly bounded (again
  using item (iv) of the definition of interpolation family)
  it has a subsequence converging pointwise
  to some $f$ which then satisfies the conclusion of the Theorem
  \ref{thm:main}

\section{The case of arbitrary $X$}
 \label{subsec:infiniteX}
  The passage from finite $X$ to infinite $X$ involves 
  a Zorn's Lemma argument.

  Let $\mathcal K$ denote a given interpolation family
  on a set $X$. 
  Let $Y$, a  finite subset of $X$, $g:Y\to\mathbb C$
  and $\rho>0$  such that for each $k\in \mathcal K$ the
  kernel
 \begin{equation*}
   Y  \times Y \ni (x,y) \mapsto (\rho^2 - g(x)g(y)^*)k(x,y)
 \end{equation*}
  is positive semi-definite, be given.

  Consider the collection $\mathcal S$ of pairs $(U,f)$ where 
  $Y\subset U \subset X$,
  $f:U\to \mathbb C$, $f|_U=g$, and  for each $k\in \mathcal K$ the
  kernel
 \begin{equation*}
    U\times U \ni (x,y) \mapsto (\rho^2 - f(x)f(y)^*)k(x,y)
 \end{equation*}
  is positive semi-definite.
  
  Partially order $\mathcal S$ as follows. Say $(U,f)\le (W,h)$
  if $U\subset W$ and $h|_U=f$.  Suppose $\mathcal C=\{(U,f_U)\}$
  is a well ordered chain from $\mathcal S$.  To see that
  $\mathcal C$ has an upper bound, let $T=\cup U$
  and define $h:T\to\mathbb C$ by $h(x)=f_U(x)$, where
  $(U,f_U)$ is any element of $\mathcal C$ for which
  $x\in U$.  The fact that $\mathcal C$ is linearly
  ordered implies that $h$ is well defined. Further,
  if $F$ is any finite subset of $T$, then there
  exists a $(U,f_U)\in\mathcal C$ such that $F\subset U$
  and hence, for each $k\in\mathcal K$, the matrix
 \begin{equation*}
  \begin{split}
   A_{k,F}=&\begin{pmatrix} (\rho^2-f_U(x)f_U(y)^*)k(x,y) \end{pmatrix}_{x,y\in F}\\
     =& \begin{pmatrix}(\rho^2-h(x)h(x)^*)k(x,y) \end{pmatrix}
  \end{split}
 \end{equation*}
  is positive semi-definite. It follows that $(T,h)\in\mathcal S$
  and is an upper bound for $\mathcal C$. 

  By Zorn's Lemma, $\mathcal C$ has a maximal element $(W,h).$
  Suppose $W\ne X$.  In this case, there is a point $z\in X\setminus W$.
  Given a finite subset $F\subset Y$, let $G=F\cup\{z\}.$
  For each $u\in \mathbb C$, define a function $q:G\to \mathbb C$ 
  by declaring $q|_F =h|_F$ and $q(z)=u$.  
  Now define $C_F$ to be the set of 
   $u\in\mathbb C$ for which the kernel
 \begin{equation*}
   G\times G \mapsto (\rho^2 -g(x)g(y)^*)k(x,y)
 \end{equation*}
   is positive semidefinite for all $k\in K.$
   The set $C_F$ is nonempty by the finite case
  of Theorem \ref{thm:main} and is also closed.
  It is bounded by condition (iv) of Definition \ref{def:kernel-family}.
  Thus $C_F$ is compact.

  The collection $\{C_F:F\subset X, \ \ |F|<\infty\}$ has
  the finite intersection property and hence there is
  a $u_*$ such that 
 \begin{equation*}
   u_* \in \cap \{C_F:F \subset X, \ \ |F|<\infty\}.
 \end{equation*}
  Define $h_*:Y\cup \{z\} \to \mathbb C$ by $h_*|_Y=h$
  and $h_*(z)=u_*$. Then $(W\cup\{z\},h_*)\in\mathcal C$
  and is greater than $(W,h)$, a contradiction which
  completes the proof.

%%%%%%%%%%%%% begins at previous section

\section{Scalar Interpolation}
 \label{sec:scalar}
   Let $\mathcal K$ be an Agler interpolation family of
   kernels on a set $X$. Let $\mathcal K_*$ denote those
   scalar kernels $k$  on $X$ which have the form
 \begin{equation*}
  k(x,y) = G(y)^* K(x,y)G(x)
 \end{equation*}
  for some $N$, kernel $K\in\mathcal K_N$, and function
  $G:X\to \mathbb C^N$.
   The following lemma says, under the 
  the conditions of equations \ref{eq:thm-main-Y}
   and \ref{eq:thm-main-X}
   in the statement of Theorem \ref{thm:main}, that 
   $\mathcal K$ can be replaced by $\mathcal K_*$.

 \begin{lemma}
   If $Y$ is a subset of $X,$ $\rho>0$, and $f:Y\to \mathbb C$, then
   the kernel 
  \begin{equation*}
    Y\times Y \ni (x,y) \mapsto (\rho^2 - f(x)f(y)^*)k(x,y)
  \end{equation*} 
   is positive semi-definite for every $k\in\mathcal K$ if
   and only if the kernel
  \begin{equation*}
    Y\times Y \ni (x,y) \mapsto (\rho^2 - f(x)f(y)^*)k_*(x,y)
  \end{equation*} 
   is positive for every $k_*\in\mathcal K_*$. 
 \end{lemma}
 
 \begin{proof}
    Fix $k\in\mathcal K$ and 
    a finite subset $F\subset X$ and consider the block matrix 
  \begin{equation*}
     A = ((\rho^2 - f(x)f(y)^*)k(x,y))_{x,y\in F}.
  \end{equation*}
    Thus $A$ is a matrix with $n\times n$ matrix entries. 
    Given a function $H:F\to \mathbb C^n$ viewed as a vector,
  \begin{equation*}
   \begin{split}
     \langle AH,H\rangle = & \sum \langle A_{x,y}H(y),H(x) \rangle \\
       = & \langle [(\rho^2-f(x)f(y)^*) H(y)^* k(x,y)H(y)]_{x,y\in F}
           o(y),o(x) \rangle,
   \end{split}
  \end{equation*}
    where $o:F\to\mathbb C$ is the constant function $o(x)=1$. 
    Hence, if 
 \begin{equation*}
    F\times F\ni (x,y)\mapsto (\rho^2-f(x)f(y)^*) H(y)^* k(x,y)H(y)
 \end{equation*}
     is positive semi-definite for each $H$, then $A$ is positive
   semi-definite.
 \end{proof}

\section{Examples: the disc and the annulus}
 \label{sec:example}
  For the case of the disc, let $\mathcal K_n=\{s_n=I_n\otimes s\},$
  where $I_n$ is the identity $n\times n$ matrix and $s$
  is Szeg\H{o}'s kernel. Given a unit vector $\gamma\in\mathbb C^n$ 
  and $\lambda\in\mathbb D$ let $Q=I-\gamma\gamma^*,$ and let
  $\varphi_\lambda$ denote a M\"obius map of the disc
  sending $\lambda$ to $0$, and $G=\varphi_\lambda \gamma\gamma^*+Q$.
  It is readily verified that
 \begin{equation*}
  \begin{split}
    k^\prime(z,w) = & s_n(z,w)-\frac{s_n(z,\lambda)\gamma\gamma^* s_n(\lambda,w)}
      {\gamma^* s_n(\lambda,\lambda)\gamma} \\
     = & G(w)^* s_n(z,w)G(z).
  \end{split}
 \end{equation*}
  Hence $\mathcal K$ is an Agler interpolation family.

  Let $\mathbb A$ denote an annulus, $\{r<|z|<\frac{1}{r}\}$.
  There is a family $k_t(z,w)$ of scalar kernels parameterized
  by $T$ in the unit circle $\mathbb T$ which collectively play
  a role on the annulus similar to that played by  Szeg\H{o}'s
  kernel on the disc \cite{Sarasonannulus}.  These are
  the kernels appearing in Abrahamse's interpolation
  theorem on $\mathbb A$ \cite{Ab}. It turns out that
  given $t\in \mathbb T$ and $\lambda\in\mathbb A$ there is an $s\in\mathbb T$
  (which can be explicitly described in terms of the Abel-Jacobi
  map) and an analytic function $\varphi_\lambda$ such that
 \begin{equation*}
  k_t(z,w) - \frac{k_t(z,\lambda)k_t(\lambda,w)}{k_t(\lambda,\lambda)}
   =\varphi_\lambda(w)^* k_s(z,w)\varphi_\lambda(z).
 \end{equation*} 
   Moreover, to each $t$ and $s$ there is a $\lambda$ such that
   the above identity holds, explaining, at least heuristically,
   the need to consider
   the whole Sarason collection of kernels when interpolating
   on $\mathbb A$.

  Let $\mathcal K_n$ denote the collection of kernels
  of the form $k_{t_1}\oplus\dots \oplus k_{t_n}$. The results
  in \cite{AD} show that $\mathcal K=(\mathcal K_n)$
  is an Agler interpolation family on $\mathbb A$. Moreover,
  interpolation with respect to this family is 
  interpolation in $H^\infty(\mathbb A)$ as in \cite{Ab}.

  As a final remark, note that in the proof of Lemma \ref{lem2}
  and using the notations there 
  if $k$ is a direct sum of kernels and if 
  $\mathcal G_{\tilde{Y}} = \mathcal L,$
  then $\kappa$ is also the direct sum of
  scalar kernels.  If this were always the case, then
  there would be no need to consider direct sums 
  in the definition of interpolation family. 
  Thus, the fact that, for scalar interpolation on a multiply
  connected domain it suffices to consider scalar kernels
  only represents additional structure not modeled by
  Theorem \ref{thm:main}.

\iffalse %%%%%%%% 
  The operators
  of multiplication by $z$ on $H^2(k_t)$ are (rank one)
  bundle shifts over $\mathbb A$; i.e., they are (cyclic) pure
  subnormal operators whose minimal normal extension
  has spectrum in the boundary of $\mathbb A$ \cite{AD}.
  
  Let $\mathcal K_n$ denote the collection of $M_n$-valued kernels $k$
  on $\mathbb A$ such that multiplication by $z$ on $H^2(k)$
  is a bundle shift. If $k\in\mathcal K$ and $k^\prime$
  is a compression of $k$ as in condition (ii) of the definition
  of interpolation family, then multiplication by $z$ on
  $H^2(k^\prime)$ is also a bundle shift and thus $k^\prime \in\mathcal K_n$.
  It follows that $\mathcal K$ is an interpolation family.

  Further, by a result in \cite{AD},
  if $k\in\mathcal K_n$, then, up to unitary equivalence
  the bundle shift determined by $H^2(k)$, is a direct
  sum of rank one bundle shifts and thus it suffices
  to very the hypothesis of Theorem \ref{thm:main} for
  scalar kernels in $\mathcal K$; i.e., only of $\mathcal K_1$.
 \fi %%%%%%%%%%%%

\end{document}